\documentclass[a4paper,11pt]{amsart}          
\usepackage{amsfonts,amsmath,latexsym,amssymb} 
\usepackage{amsthm}      
\usepackage[dvipsnames]{xcolor}          
\usepackage{mathrsfs,upref}         
\usepackage{mathptmx}		    
\usepackage{float}
\usepackage{tikz}
\usepackage{circuitikz}
\usepackage[colorlinks=true]{hyperref}
\hypersetup{urlcolor=blue, citecolor=red}
\usepackage{mathtools}

\newtheorem{theorem}{Theorem}           
\newtheorem{lemma}[theorem]{Lemma}

\theoremstyle{definition}

\newtheorem{remark}[theorem]{Remark}
\newtheorem{proposition}[theorem]{Proposition}

\DeclareMathOperator\trace{trace}
\renewcommand{\Re}{\operatorname{Re}}
\renewcommand{\Im}{\operatorname{Im}}

\newcommand{\R}{\mathbb{R}}
\newcommand{\C}{\mathbb{C}}
\newcommand{\cfreq}{s}

\author[Anna Muranova]{Anna Muranova*}
\author[Robert Schippa]{Robert Schippa}
\address{Anna Muranova: 5050 Institut f\"ur discrete Mathematik, Steyrergasse 30/III,
8010 Graz, Austria} 

 \email{anna.muranova@gmail.com, muranova@math.tugraz.at}
 \address{Robert Schippa: Fakult\"{a}t f\"{u}r Mathematik, Karlsruher Institut f\"{u}r Technologie, Englerstrasse 2, 76131 Karlsruhe, Germany} 
 \email{robert.schippa@kit.edu}

\title[Complex Laplacian]{Eigenvalues of the normalized complex Laplacian on finite electrical networks}
\thanks{*Corresponding author}

\begin{document}

\maketitle



\begin{abstract}
The spectrum of the normalized complex Laplacian for electrical networks is analyzed. We show that eigenvalues lie in a larger region compared to the case of the real Laplacian. We show the existence of eigenvalues with negative real part and absolute value greater than $2$. An estimate from below for the first non-vanishing eigenvalue in modulus is provided. We supplement the estimates with examples, showing sharpness.
\end{abstract}

{\footnotesize
\smallskip

}

\section{Introduction}

The purpose of this note is to analyze the complex Laplacian for a class of finite weighted graphs. In the following we consider finite \emph{electrical networks} (cf. \cite{Muranova3,Muranova1,MuranovaThesis}). These consist of finite sets of vertices and edges $(V,E)$, where loops are excluded. For any edge $xy$, $x,y \in V$, there are non-negative real numbers $L_{xy}$, $D_{xy}$, $R_{xy}$ satisfying
\begin{equation*}
L_{xy} + D_{xy} + R_{xy} > 0.
\end{equation*}
For $s \in \C$ with $\Re s > 0$, we consider the finite weighted graph $(V,E)$ with weights given by the \emph{admittances}
\begin{equation}
\label{eq:admittance}
\rho^{(s)}(x,y) = \frac{s}{L_{xy}s^2 + R_{xy} s + D_{xy}}.
\end{equation}
Since $s$ will be clear from context, we write $\rho_{xy}$ for $\rho^{(s)}(x,y)$ for the sake of brevity.
We set $\rho_{xy}\equiv 0$ if $xy$ is not an edge. Then, any electrical network is uniquely determined by the pair $(V,\rho)$.
 Let
\begin{equation*}
\rho(x) = \sum_{y\in V} \rho_{xy}.
\end{equation*}
%
%
The admittance in form \eqref{eq:admittance} corresponds to the case of electrical networks with passive elements (coils, capacitors, and resistors). In this case $\cfreq$ corresponds to a complex frequency  (cf. \cite{Brune, Feynman1, Feynman2}).
We consider the normalized complex-weighted Laplacian:
\begin{equation}\label{Delta}
\widetilde \Delta_\rho f(x)=f(x)-\dfrac{1}{\rho(x)}\sum_{y\in V} f(y)\rho_{xy}.
\end{equation}
Set ${\rho_{xy}=\tau_{xy}+i\sigma_{xy}}$, where $\tau_{xy}, \sigma_{xy} \in \R$ and ${\rho(x)=\tau(x)+i\sigma(x)}$. Clearly, $\tau(x)=\sum_{y}\tau_{xy}$ and $\sigma(x)=\sum_{y}\sigma_{xy}$. 

Since admittance is positive real function (cf. \cite{Brune, Muranova3, MuranovaThesis}), i.e., $\Re \cfreq >0$, whenever $\rho^{(s)}_{xy}>0$ (in particular, $\rho^{(s)}_{xy}$ has no poles), the normalized complex-weighted Laplacian is well-defined. For the sake of completeness, we show this in Proposition \ref{prop:Rerho}.

The real Laplacian ($L_{xy}=D_{xy}=0$ for any $x,y$)  arises in both electrical networks with resistors and random walks (cf. \cite{DoyleSnell, Grigoryan, Levin, Woess}). The complex Laplacian arises in AC electrical networks and corresponds to Kirchoff's law (cf. \cite{Feynman1, Hughes, Muranova1, MuranovaThesis}). Note that if $s \in \R^+$, then the resulting Laplacian is still real and corresponds to some random walk.

For $s \in \R^+$ the resulting Laplacian is a self-adjoint, non-negative operator acting on a Hilbert space with scalar product depending on $\rho$. The real Laplacian is studied extensively (cf. \cite{ BauerJost, Chung, Grigoryan}). In this case, the eigenvalues lie in the interval $[0,2]$. In Section \ref{section:Preliminaries} we recall among basic facts how $2$ is attained as an eigenvalue for bipartite graphs.

In Section \ref{section:EigenvalueRegions} we show the following generalization:
\begin{theorem}
\label{thm:EigenvaluesAroundOne}
Let $\lambda$ denote an eigenvalue of \eqref{Delta}. Then, we find the following estimate to hold:
\begin{equation}
\label{eq:EigenvalueEstimateII}
|1 - \lambda| \leq \left( \frac{|s|}{\Re s} \right).
\end{equation}
\end{theorem}

Secondly, we show that eigenvalues are additionally confined to the region described below:
\begin{theorem}
\label{thm:CircleRegions}
Let $\Delta_\rho$ be a normalized complex-weighted Laplacian. Then the following holds:
\begin{itemize}
\item
All its eigenvalues with positive real part lie in the circle with center at $(1,|\Im s|/\Re s)$ and radius $\sqrt {1+(\Im s/\Re s)^2}$.
\item
All its eigenvalues with negative real part lie in the circle with the center at $(1,-|\Im s|/\Re s)$ and radius $\sqrt {1+(\Im s/\Re s)^2}$.
\item
All its real eigenvalues lie in $[0,2]$.
\end{itemize}
\end{theorem}
We shall see that for the linear graph $P_4$ a particular choice of weights exhausts the radii of the circles from the above theorem.

In Section \ref{section:SmallestEigenvalue} we prove the following bound on the smallest eigenvalue in modulus:
\begin{theorem}
\label{thm:Diameter}
Let $\lambda_1$ denote the smallest eigenvalue in modulus for the complex Laplacian $\widetilde \Delta_\rho$, and let $D$ denote the diameter of the underlying graph. Let
\begin{align*}
C_1 = \min_x \sum_y \big( \frac{1}{L_{xy} + R_{xy} + D_{xy}} \big), \quad C_2 = \sum_{x,y} \frac{1}{L_{xy} + R_{xy} + D_{xy}}
\end{align*}
and suppose that 
\begin{equation}
\label{eq:LowerBoundRealpart}
C_1 \Re s \min(|s|^2,|s|^{-4}) - C_2 \frac{|\Im s|}{\Re s} \frac{\max(1,|s|^2)}{\Re s} > 0.
\end{equation}
Then, we find the following estimate to hold:
\begin{equation*}
|\lambda_1| \geq \frac{C_1 (\Re s)^2}{D \cdot C_2 \cdot \min(1,|s|^2)}.
\end{equation*}
\end{theorem}

\emph{Outline of the paper.} In Section \ref{section:Preliminaries} we recall basic identities, which are needful for the proofs in later sections. Furthermore, we record symmetries of the problem. In Section \ref{section:EigenvalueRegions} we prove Theorems \ref{thm:EigenvaluesAroundOne} and \ref{thm:CircleRegions}. The eigenvalues for the linear graph $P_4$ are computed for a particular choice of weight. In Section \ref{section:SmallestEigenvalue} we prove Theorem \ref{thm:Diameter}.

\section{Preliminaries}
\label{section:Preliminaries}
We start with the following repetition concerning the real Laplacian:
\begin{proposition}
Let $(V,E)$ be a finite connected graph with real weights and $\widetilde \Delta_\rho$ be the real Laplacian. Then, we find the following to hold:
\begin{itemize}
\item[1.] The eigenvalues lie in the interval $[0,2]$ of $\widetilde \Delta_\rho$, where $0$ is a simple eigenvalue.
\item[2.] $2$ is an eigenvalue if and only if $V$ is \emph{bipartite}, i.e., there is a partition of vertices $V = V_+ \cup V_-$, $V_+ \cap V_- = \emptyset$ such that for any $x \in V_+$ $x \sim y$ implies that $y \in V_-$.
\item[3.] $V$ is bipartite if and only if for any eigenvalue $\lambda$ we find $2-\lambda$ to be another eigenvalue.
\end{itemize}
\end{proposition}

Recall the following Green's formula (\cite{Muranova1, MuranovaThesis}):
\begin{equation}\label{Green}
\sum_{x\in V}\widetilde \Delta_\rho f(x)g(x)\rho(x)=\dfrac{1}{2}\sum_{x,y\in V}(\nabla_{xy}f)(\nabla_{xy}g)\rho_{xy},
\end{equation}
where we denote
\begin{equation*}
\nabla_{xy}f=f(y)-f(x).
\end{equation*}
Applying Green's formula to $g=\overline f$ (complex conjugate), we get
\begin{equation}\label{energy}
\sum_{x\in V}\widetilde \Delta_\rho f(x)\overline{f(x)}\rho(x)=\dfrac{1}{2}\sum_{x,y\in V}|\nabla_{xy}f|^2\rho_{xy}.
\end{equation}
Note, that the right hand side of the last equation corresponds to the \emph{complex power} (cf. \cite{AlonsoRuiz, Baez, ChenTeplyaev, MuranovaThesis}).
For estimates on the eigenvalues we will need estimates for $\rho_{xy}$. These are collected in the following proposition:
\begin{proposition}\label{prop:Rerho}
Let $\rho_{xy}$ be as in \eqref{eq:admittance} with the notations from above. Then, the following holds:
\begin{itemize}
\item[1.] If $x$ and $y$ are related, then $\Re \rho_{xy} > 0$. 
\item[2.] 
\begin{equation}
\label{eq:EstimateModulusAdmittance}
| \rho_{xy} | \leq \frac{|s|}{\Re s} \Re \rho_{xy}.
\end{equation}
\item[3.]
\begin{equation}
\label{eq:EstimateModulusRho}
| \rho(x)| \leq \sum_y \frac{1}{L_{xy} + D_{xy} + R_{xy}} \frac{\max(1,|s|^2)}{\Re s} .
\end{equation}
\end{itemize}
\end{proposition}

\begin{proof}
\begin{itemize}
\item[1.]
Note that for any complex number $z\in \Bbb C$, $\Re z>0$ if and only if $\Re \frac{1}{z}>0$. We have
\begin{equation*}
\Re \frac{1}{\rho_{xy}}=L_{xy}\Re s+R_{xy}+D_{xy}\Re\frac{1}{s}>0, \mbox{ whenever }\Re s>0.
\end{equation*}
Therefore,  $\Re \rho_{xy} > 0$. 
\item[2.] 
\begin{align*}
\dfrac{\Re\rho_{xy}}{|\rho_{xy}|}=&|\rho_{xy}|\Re \frac{1}{\rho_{xy}}=\dfrac {1}{\left|L_{xy} s+R_{xy}+\frac{D_{xy}}{s}\right|}\left(L_{xy}\Re s+R_{xy}+D_{xy}\Re\frac{1}{s}\right)\\
=&\dfrac {1}{\left|L_{xy} s+R_{xy}+\frac{D_{xy}}{s}\right|}\left(L_{xy}\Re s+R_{xy}+D_{xy}\frac{\Re s}{|s|^2}\right)\\
\geq&\dfrac {1}{L_{xy} |s|+R_{xy}+\frac{D_{xy}}{|s|}}\left(L_{xy}\Re s+R_{xy}+D_{xy}\frac{\Re s}{|s|^2}\right)
\\
\geq&\dfrac {1}{L_{xy} |s|+R_{xy}+\frac{D_{xy}}{|s|}}\left(L_{xy}\Re s+R_{xy}\frac{\Re s}{|s|}+D_{xy}\frac{\Re s}{|s|^2}\right)=\dfrac{\Re s}{|s|},
\end{align*}
from which \eqref{eq:EstimateModulusAdmittance} follows.
\item[3.]
\begin{align*}
| \rho(x)| =&\left|\sum_y \rho_{xy}\right|\leq \sum_y |\rho_{xy}|\leq \sum_y \dfrac {1}{\left|L_{xy} s+R_{xy}+\frac{D_{xy}}{s}\right|}\leq \sum_y \dfrac {1}{\Re \left(L_{xy} s+R_{xy}+\frac{D_{xy}}{s}\right)}\\
=& \sum_y \dfrac {1}{ \left(L_{xy}\Re s+R_{xy}+D_{xy}\frac{\Re s}{|s|^2}\right)}=  \sum_y \dfrac {1}{ \Re s\left(L_{xy}+R_{xy}\frac{1}{\Re s}+D_{xy}\frac{1}{|s|^2}\right)}\\
\leq& \sum_y \dfrac {1}{ \Re s\left(L_{xy}+R_{xy}+D_{xy}\right)\min \left(\frac{1}{\Re s}, \frac{1}{|s|^2},1\right)}
= \sum_y \dfrac {\max \left({\Re s}, {|s|^2},1\right)}{ \Re s\left(L_{xy}+R_{xy}+D_{xy}\right)}\\
=& \sum_y \dfrac {\max \left({|s|^2},1\right)}{ \Re s\left(L_{xy}+R_{xy}+D_{xy}\right)},
\end{align*}
where in the last line we have used the fact that either $\Re s\le 1$ or $1<\Re s<(\Re s)^2\le |s|^2$.
\end{itemize}
\end{proof}

\begin{proposition}
$0$ is a simple eigenvalue of the complex Laplacian.
\end{proposition}
\begin{proof}
It is clear that $f \equiv 1$ is an eigenfunction with eigenvalue $0$. On the other hand, for an eigenfunction $f$ with eigenvalue $0$, \eqref{energy} yields
\begin{equation*}
0 = \frac{1}{2} \sum_{x,y \in V} |\nabla_{xy} f|^2 \rho_{xy}.
\end{equation*}

Since $\Re \rho_{xy} >0$ provided that $x$ and $y$ are related, taking the real part of the above display we find that $|\nabla_{xy} f| = 0$ whenever $x$ and $y$ are related. Since the graph is connected, $f \equiv c$. The proof is complete.
\end{proof}
\begin{proposition}
The sum of the all eigenvalue of Laplacian, counted with algebraic multiplicities, is equal to the number of vertices in the graph, i.e
\begin{equation}\label{sumofev}
\lambda_0+\lambda_1+\dots+\lambda_{n-1}=n,
\end{equation}
where $n=|V|$.

Moreover, 
\begin{equation}\label{estimate::reOfeigenvalue}
\max_i \Re \lambda_{i}\ge \dfrac{n}{n-1}>1,
\end{equation}
and
\begin{equation*}
\min_i \Re \lambda_{i}\le \dfrac{n}{n-1}.
\end{equation*}
\end{proposition}

\begin{proof}
The first fact follows from the consideration of the trace of the Laplacian matrix $A$. Indeed, since $A_{ii} = 1$, $i=1,\ldots,n$, we find
\begin{equation*}
\lambda_0+\lambda_1+\dots+\lambda_{n-1}=\trace A=n,
\end{equation*}
where $n=|V|$.

Further, since $\lambda_0=0$, we have

\begin{equation*}
n=\Re \lambda_1+\dots+\Re \lambda_{n-1}\le (n-1)\max_i \Re \lambda_{i}.
\end{equation*}
and 
\begin{equation*}
n=\Re \lambda_1+\dots+\Re \lambda_{n-1}\ge (n-1)\min_i \Re \lambda_{i}.
\end{equation*}
Therefore, 
\begin{equation*}
\max_i \Re \lambda_{i}\ge \dfrac{n}{n-1}>1
\end{equation*}
and
\begin{equation*}
\min_i \Re \lambda_{i}\le \dfrac{n}{n-1}.
\end{equation*}
\end{proof}

\begin{remark}
Note that if we assume that 
\begin{equation*}
0=|\lambda_0|<|\lambda_1|\le|\lambda_2|\le\dots\le|\lambda_{n-1}|,
\end{equation*}
then from \eqref{estimate::reOfeigenvalue} follows that
\begin{equation*}
|\lambda_{n-1}|\ge \dfrac{n}{n-1}.
\end{equation*}
Moreover, from \eqref{sumofev} follows that
\begin{equation*}
\Im \lambda_1+\Im\lambda_2+\dots+\Im \lambda_{n-1}=0,
\end{equation*}
and, therefore,
\begin{equation*}
\min_i\Im \lambda_i\le 0\le \max_i\Im \lambda_i.
\end{equation*}
\end{remark}
Next, we show that switching to the dual network conjugates eigenvalues and eigenvectors:
\begin{proposition}
If we consider the dual network, i.e., the network with the weights $\overline \rho_{xy}$, eigenvalues (eigenvectors) of its complex Laplacian will be conjugated to the corresponding eigenvalues (eigenvectors) of the Laplacian of the original network. 
\end{proposition}
\begin{proof}
Let $\lambda f(x)=\widetilde \Delta_\rho f(x)$.
Then
\begin{align*}
\overline \lambda \overline {f(x)}=&\overline{\widetilde\Delta_\rho f(x)}=\overline{f(x)}-\dfrac{1}{\overline{\rho(x)}}\sum_{y}\overline{ f(y)}\overline\rho_{xy}=\widetilde\Delta_{\overline \rho} \overline{f(x)}.
\end{align*}
\end{proof}

The claims about eigenvalues for bipartite graphs in the real case generalize as follows:
\begin{proposition}
\label{prop:ComplexEigenvaluesBipartiteGraphs}
With the above notations, let $\widetilde \Delta_{\rho}$ be the complex Laplacian of an electrical network. Suppose that the underlying graph is bipartite. Then for any eigenvalue $\lambda$ we find $2-\lambda$ to be another eigenvalue.
\end{proposition}

\begin{proof}
If $V_+$ and $V_-$ are suitable partitions of the graph, and $(\lambda,f)$ is an eigenpair, then the function $g$, given by
\begin{equation*}
g(x)=\begin{cases}
f(x), x\in V_+,\\
-f(x), x\in V_-,
\end{cases}
\end{equation*}

is the eigenfuction, corresponding to the eigenvalue $2-\lambda$. Indeed, for $x\in V_+$,
\begin{align*}
{\widetilde\Delta_\rho g(x)}=&{g(x)}-\dfrac{1}{{\rho(x)}}\sum_{y}{ g(x)}\rho_{xy}={f(x)}+\dfrac{1}{{\rho(x)}}\sum_{y}{ f(x)}\rho_{xy}\\
=&2{f(x)}-\left(f(x)-\dfrac{1}{{\rho(x)}}\sum_{y}{ f(x)}\rho_{xy}\right)\\
=&2f(x)-{\widetilde\Delta_\rho f(x)}=(2-\lambda)f(x)=(2-\lambda)g(x).
\end{align*}
The case $x\in V_-$ can be treated analogously.
\end{proof}

\section{Eigenvalue regions}
\label{section:EigenvalueRegions}
\subsection{Proof of Theorems \ref{thm:EigenvaluesAroundOne} and \ref{thm:CircleRegions}}
Firstly, we show Theorem \ref{thm:EigenvaluesAroundOne}:
\begin{proof}[Proof of Theorem \ref{thm:EigenvaluesAroundOne}]
We choose an eigenfunction $f$ of $\widetilde \Delta_{\rho}$ with $|f(y)| \leq 1$ for any $y \in V$ and $\max_ x|f(x)| = 1$.
By \eqref{eq:EstimateModulusAdmittance}, we find
\begin{align*}
|1 - \lambda| |f(x)| &\leq \left| \sum_y \frac{\rho_{xy}}{\rho(x)} f(y) \right| \\
&\leq \sum_y  \frac{|\rho_{xy}|}{|\rho(x)|} |f(y)| \\
&\leq \frac{|s|}{\Re s} \sum_y \frac{\Re \rho_{xy}}{|\rho(x)|} = \frac{|s|}{\Re s} \frac{\Re \rho(x)}{|\rho(x)|}\le  \frac{|s|}{\Re s}.
\end{align*}
\end{proof}
Note how \eqref{eq:EigenvalueEstimateII} generalizes the claim for real $s$. The eigenvalues of the Laplacian in the real case lie in the interval $[0,2]$; hence, the estimate is clearly sharp by comparison with the real case.

Next, we prove Theorem \ref{thm:CircleRegions}:
\begin{proof}[Proof of Theorem \ref{thm:CircleRegions}]
Let ${\lambda=u+iw}$ be an eigenvalue of the normalized complex-weighted Laplacian $\widetilde \Delta_\rho$ with the eigenfunction $f$, i.e, 
\begin{equation*}
\lambda f(x)=\widetilde \Delta_\rho f(x).
\end{equation*}
Then by Green's formula we have
\begin{equation}\label{eveq}
\sum_{x\in V}\lambda|f(x)|^2\rho(x)=\dfrac{1}{2}\sum_{x,y\in V}|\nabla_{xy}f|^2\rho_{xy}.
\end{equation}
Let us write separately real and imaginary parts of the last equality, assuming  ${\rho_{xy}=\tau_{xy}+i\sigma_{xy}}$ and ${\rho(x)=\tau(x)+i\sigma(x)}$. Note that in this case $\tau(x)=\sum_{y}\tau_{xy}$ and $\sigma(x)=\sum_{y}\sigma_{xy}$. 

Then,
\begin{equation*}
\lambda\sum_{x\in V}|f(x)|^2\rho(x)=(u+iw)\sum_{x\in V}|f(x)|^2(\tau(x)+i\sigma(x)),
\end{equation*}
and
\begin{equation*}
\dfrac{1}{2}\sum_{x,y\in V}|\nabla_{xy}f|^2\rho_{xy}=\dfrac{1}{2}\sum_{x,y\in V}|\nabla_{xy}f|^2(\tau_{xy}+i\sigma_{xy}).
\end{equation*}

Therefore, for the real part of \eqref{eveq}, we have
\begin{equation}\label{realpart}
u\sum_{x\in V}|f(x)|^2\tau(x)-w\sum_{x\in V}|f(x)|^2\sigma(x)=\dfrac{1}{2}\sum_{x,y\in V}|\nabla_{xy}f|^2\tau_{xy},
\end{equation}
and for the imaginary part,
\begin{equation}\label{imaginarypart}
u\sum_{x\in V}|f(x)|^2\sigma(x)+w\sum_{x\in V}|f(x)|^2\tau(x)=\dfrac{1}{2}\sum_{x,y\in V}|\nabla_{xy}f|^2\sigma_{xy}.
\end{equation}
Multiplying \eqref{realpart} by $u$ and \eqref{imaginarypart} by $w$ and summing the obtained equalities up, we get

\begin{equation}\label{eq1}
(u^2+w^2)\sum_{x\in V}|f(x)|^2\tau(x)=\dfrac{1}{2}\sum_{x,y\in V}|\nabla_{xy}f|^2(\tau_{xy}u+\sigma_{xy}w).
\end{equation}
Let us estimate the left-hand side, using the obvious inequality 
\begin{equation*}
|\nabla_{xy}f|^2=|f(y)-f(x)|^2\le 2(|f(y)|^2+|f(x)|^2).
\end{equation*}
Moreover, the right-hand side of \eqref{eq1} is positive, this means also the left-hand side is positive and at least one $(\tau_{xy}u+\sigma_{xy}w)$ is positive.

Let $w>0$.
Then, we have
\begin{align*}
(u^2+w^2)\sum_{x\in V}|f(x)|^2\tau(x)&=\dfrac{1}{2}\sum_{x,y\in V}|\nabla_{xy}f|^2\tau_{xy}(u+\dfrac{\sigma_{xy}}{\tau_{xy}}w)\\
&\le \sum_{x,y\in V}(|f(y)|^2+|f(x)|^2)\tau_{xy}(u+\dfrac{\sigma_{xy}}{\tau_{xy}}w) \\
&\le \sum_{x,y\in V}(|f(y)|^2+|f(x)|^2)\tau_{xy}(u+w\max_{x\sim y}\dfrac{\sigma_{xy}}{\tau_{xy}}) \\
&\le 2\sum_{x\in V}|f(x)|^2\tau(x)(u+w\max_{x\sim y}\dfrac{\sigma_{xy}}{\tau_{xy}}) .
\end{align*}
Therefore,
\begin{align*}
u^2+w^2&\le 2(u+w\max_{x\sim y}\dfrac{\sigma_{xy}}{\tau_{xy}}) ,
\end{align*}
i.e.,
\begin{align*}
(u-1)^2+\left(w-\max_{x\sim y}\dfrac{\sigma_{xy}}{\tau_{xy}}\right)^2&\le 1+\left(\max_{x\sim y}\dfrac{\sigma_{xy}}{\tau_{xy}}\right)^2.
\end{align*}

Let $w<0$.
Then, we have
\begin{align*}
(u^2+w^2)\sum_{x\in V}|f(x)|^2\tau(x)&=\dfrac{1}{2}\sum_{x,y\in V}|\nabla_{xy}f|^2\tau_{xy}(u+\dfrac{\sigma_{xy}}{\tau_{xy}}w)\\
&\le \sum_{x,y\in V}(|f(y)|^2+|f(x)|^2)\tau_{xy}(u+\dfrac{\sigma_{xy}}{\tau_{xy}}w) \\
&\le \sum_{x,y\in V}(|f(y)|^2+|f(x)|^2)\tau_{xy}(u+w\min_{x\sim y}\dfrac{\sigma_{xy}}{\tau_{xy}}) \\
&\le 2\sum_{x\in V}|f(x)|^2\tau(x)(u+w\min_{x\sim y}\dfrac{\sigma_{xy}}{\tau_{xy}}) .
\end{align*}
Therefore,
\begin{align*}
u^2+w^2&\le 2(u+w\min_{x\sim y}\dfrac{\sigma_{xy}}{\tau_{xy}}) ,
\end{align*}
i.e.,
\begin{align*}
(u-1)^2+\left(w-\min_{x\sim y}\dfrac{\sigma_{xy}}{\tau_{xy}}\right)^2&\le 1+\left(\min_{x\sim y}\dfrac{\sigma_{xy}}{\tau_{xy}}\right)^2.
\end{align*}

Let $w=0$.
Then, we have
\begin{align*}
u^2\sum_{x\in V}|f(x)|^2\tau(x)&=\dfrac{1}{2}\sum_{x,y\in V}|\nabla_{xy}f|^2\tau_{xy}u\\
&\le 2\sum_{x\in V}|f(x)|^2\tau(x)u .
\end{align*}
Therefore,
\begin{align*}
u^2&\le 2u,
\end{align*}
which means that the real eigenvalues of the normalized complex Laplacian lie in $[0,2]$.

From \eqref{eq:EstimateModulusAdmittance} follows by squaring
\begin{equation*}
\dfrac{\Re^2 \rho_{xy}}{\Re^2 \rho_{xy}+\Im^2 \rho_{xy}}\ge \dfrac{\Re^2 s}{\Re^2 s+\Im^2 s}.
\end{equation*}
This yields
\begin{equation*}
\dfrac{\Im^2 \rho_{xy}}{\Re^2 \rho_{xy}}\le \dfrac{\Im^2 s}{\Re^2 s},
\end{equation*}
and hence,
\begin{equation*}
-\left| \dfrac{\Im s}{\Re s}\right|\le \dfrac{\Im \rho_{xy}}{\Re\rho_{xy}}\le\left| \dfrac{\Im s}{\Re s}\right|.
\end{equation*}
The proof is complete.
\end{proof}

\subsection{Examples and sharpness of the estimates}

With the following example we illustrate the estimates from above and point out how we cannot scale down the radii of the circles in Theorem \ref{thm:CircleRegions}
.
We consider the linear graph $P_4$ with four vertices and edges weighted $s$, $1/s$, $s$.
The complex normalized Laplacian is given by the matrix
\begin{equation}
\label{eq:ComplexLaplacian}
A = 
\begin{pmatrix}
1 & - 1 & 0 & 0 \\
\frac{-s^2}{1+s^2} & 1 & \frac{-1}{s^2 + 1} & 0 \\
0 & \frac{-1}{s^2+1} & 1 & \frac{-s^2}{s^2+1} \\
0 & 0 & -1 & 1
\end{pmatrix}
.
\end{equation}
We find for the eigenvalues
\begin{equation}
\label{eq:EigenvaluesA4}
\frac{1}{1+s^2}, \quad 0, \quad 2, \quad \frac{1+2s^2}{1+s^2}.
\end{equation}
The corresponding eigenvectors are computed as
\begin{equation*}
(-1,\frac{-s^2}{1+s^2}, \frac{s^2}{1+s^2}, 1), \; (1,1,1,1), \; (-1,1,-1,1), \; (1,\frac{-s^2}{1+s^2},\frac{-s^2}{1+s^2}, 1).
\end{equation*}

On the Figure \ref{fig} we green circles correspond to the estimate in Theorem \ref{thm:CircleRegions}, blue circle corresponds to Theorem \ref{thm:EigenvaluesAroundOne}  (
everything is drawn for the case $s=1+2i$) and red dots are the eigenvalues $0$, $2$, $-1/10-2i/10$, $21/10+2i/10$. Note, that in this case there is an eigenvalue with real part, larger than $2$, and an eigenvalue with negative real part.

 Let $s=s_1+is_2, \; s_1,s_2 >0$. We shall show sharpness of Theorem \ref{thm:CircleRegions} considering the eigenvalue
\begin{equation*}
z = 1 + \frac{s^2}{1+s^2}.
\end{equation*}
For eigenvalues of the complex Laplacian with positive imaginary part $w$ and real part $u$ we have shown the estimate
\begin{equation*}
(u-1)^2 + (w-m)^2 \leq 1 + m^2, \quad m = \max \frac{\Im \rho_{xy}}{\Re \rho_{xy}}.
\end{equation*}
Note that in our example the maximum ratio is $m= s_2/s_1$. We show that for any $\varepsilon>0$ the estimate
\begin{equation}
\label{eq:EpsEstimate}
(u-1)^2 + (w-m)^2 \leq 1 + m^2 - \varepsilon,
\end{equation}
fails, choosing a sufficiently large value of $s_1$.

Firstly, we compute
\begin{align*}
(u-1)^2 &= \big( \frac{(s^2 + \bar{s}^2)/2 + |s|^4}{|1+s^2|^2} \big)^2 = \big( \frac{s_1^2 - s_2^2 + (s_1^2+s_2^2)^2}{|1+s^2|^2} \big)^2, \\
(w-m)^2 &= \big( \frac{2s_1 s_2}{|1+s^2|^2} - \frac{s_2}{s_1} \big)^2.
\end{align*}
Multiplying \eqref{eq:EpsEstimate} with $|1+s^2|^4$ we find
\begin{equation*}
(s_1^2 - s_2^2 + (s_1^2 + s_2^2)^2)^2 + (2s_1^2 s_2 - \frac{s_2}{s_1} |1+s^2|^2)^2 \leq |1+s^2|^4 (1+ \frac{s_2^2}{s_1^2} - \varepsilon).
\end{equation*}
Multiplying the above with $s_1^2$ gives
\begin{equation*}
s_1^2(s_1^2 - s_2^2 + (s_1^2 + s_2^2)^2)^2 + s_2^2(2s_1^2 - |1+s^2|^2)^2 \leq |1+s^2|^4 (s_2^2 + (1-\varepsilon)s_1^2).
\end{equation*}
Subtracting $s_2^2 |1+s^2|^4$ on both sides and dividing by $s_1^2$ yields
\begin{equation*}
(s_1^2 - s_2^2 + (s_1^2 + s_2^2)^2 )^2 + 4s_1^2 s_2^2 - 4s_2^2 |1+s^2|^2 \leq (1-\varepsilon) |1+s^2|^4.
\end{equation*}
We consider $s_2 \ll 1 \ll s_1$ such that the highest power of $s_1$ dominates. But on the left-hand side, we find $s_1^8$ and on the right-hand side $(1-\varepsilon) s_1^8$. Therefore, the inequality \eqref{eq:EpsEstimate} fails choosing $s_1$ large enough (for fixed $s_2$).

\begin{figure}[H]
\centering
\begin{tikzpicture}[scale=0.7]

   \path [draw=none,fill=blue!20] (1,0) circle (2.24);
   \path [draw=none,fill=green!20, semitransparent] (1,2) circle (2.24);
   \path [draw=none,fill=green!20, semitransparent] (1,-2) circle (2.24);
    \draw [->] (-2,0) -- (5,0) node [below left]  {$s_1$};
    \draw [->] (0,-5) -- (0,5) node [below right] {$s_2$};

   \node [below left,black] at (0,0) {$0$};
   \node [below,black] at (1,0) {$1$};
    \draw[black, line width = 0.20mm]  (1,-0.1) -- (1,0.1); 
   
\path [draw=none,fill=red] (0,0) circle (0.1);
\path [draw=none,fill=red] (2,0) circle (0.1);
   
\path [draw=none,fill=red] (-0.1,-0.2) circle (0.1);
\path [draw=none,fill=red] (2.1,0.2) circle (0.1);

\end{tikzpicture}
\caption{Location of eigenvalues for $s=1+2i$}
\label{fig}
\end{figure}
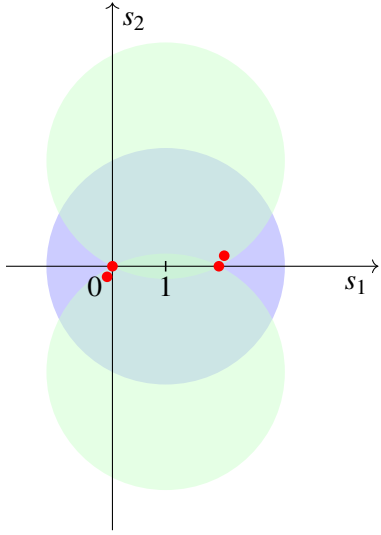

%
%
%
%
%
%
%
%

\section{Estimate for smallest eigenvalue in modulus}
\label{section:SmallestEigenvalue}
For the real normalized Laplacian the smallest positive eigenvalue (also called \emph{spectral gap}) plays a crucial role finding the mixing properties of the associated random walk. In the real case, a lower bound is provided by the \emph{Cheeger constant} (cf. \cite{Chung}). Roughly speaking, this measures the connectedness when removing vertices. Evidently, the geometry of the underlying graph plays the crucial role. Here, we do not know how to generalize the Cheeger constant to the case of complex weights because eigenspaces of eigenvalues are not necessarily orthogonal. We can still prove the lower bound depending on the diameter. We start with a lower bound on $\Re \rho_{xy}$.

\begin{lemma}
\label{lem:LowerBoundAdmittance}
Let
\begin{align*}
C_1 = \min_x \sum_y \big( \frac{1}{L_{xy} + R_{xy} + D_{xy}} \big), \quad C_2 = \sum_{x,y} \frac{1}{L_{xy} + R_{xy} + D_{xy}} .
\end{align*}
Then, we find the following estimate to hold:
\begin{equation}
\label{eq:LowerBoundRealPart}
\begin{split}
&\quad \min_x \Re (\rho(x)) - \frac{| \Im s|}{\Re s} \sum_x \Re (\rho(x)) \\
 &\geq C_1 \Re s \min(|s|^2,|s|^{-4}) - C_2 \frac{|\Im s|}{\Re s} \frac{\max(1,|s|^2)}{\Re s}.
\end{split}
\end{equation}
\end{lemma}
\begin{proof}
Let $r = \rho_{xy}$ for brevity. Note that
\begin{equation*}
\Re (r) = |r|^2 \Re \left( \frac{1}{r} \right) = |r|^2 \left( \Re (s) L_{xy} + R_{xy} + D_{xy} \frac{\Re s}{|s|^2} \right).
\end{equation*}
Moreover,
\begin{align*}
\big| \frac{1}{r} \big| &= |L_{xy} s + R_{xy} + \frac{D_{xy}}{s} | \leq L_{xy} |s| + R_{xy} + \frac{D_{xy}}{|s|} \\
&\leq (L_{xy} + R_{xy} + D_{xy}) \max( |s|, 1/|s|).
\end{align*}
Consequently,
\begin{equation*}
\Re (r) \geq \frac{1}{(L_{xy} + R_{xy} + D_{xy})^2 \max(|s|^2,1/|s|^2)} (L_{xy} + R_{xy} + D_{xy}) \min(\Re s, \Re \frac{s}{|s|^2})
\end{equation*}
because $\min(1, \Re s, \Re s/|s|^2) = \min( \Re s, \Re s/|s|^2)$, which can be seen by considering $|s| \leq 1$ and $|s| \geq 1$.
Hence,
\begin{align*}
\Re (\rho_{xy}) &\geq \frac{\Re s}{L_{xy} + D_{xy} + R_{xy}} \frac{\min(1,|s|^{-2})}{\max(|s|^2,|s|^{-2})} \\
&= \frac{\Re (s)}{L_{xy} + R_{xy} + D_{xy}} \min(|s|^2, |s|^{-4}).
\end{align*}
Hence, we find by \eqref{eq:EstimateModulusRho}
\begin{align*}
&\; \min_x \Re (\rho(x)) - \frac{| \Im s|}{\Re s} \sum_x \Re (\rho(x)) \\
&\geq \Re (s) \big[ \min_x \sum_y \big( \frac{1}{L_{xy} + R_{xy} + D_{xy}} \big) \big] \min(|s|^2,|s|^{-4}) - \frac{| \Im s|}{\Re s} \sum_x |\rho(x)| \\
&= C_1 \Re s \min(|s|^2,|s|^{-4}) - C_2 \frac{|\Im s|}{\Re s} \frac{\max(1,|s|^2)}{\Re s}.
\end{align*}
\end{proof}

Thus, for fixed $\Re s$, we can find the above quantity to be positive choosing $\Im s$ sufficiently small only depending on $\Re s$ and the electrical network. Moreover, for $s=1$ the lower bound is attained. We are ready to prove Theorem \ref{thm:Diameter}.

\begin{proof}
Let $\lambda_1$ be the smallest eigenvalue in modulus and let $f$ be a corresponding eigenfunction. We normalize $f$ such that $f(x_0) = 1$ and $\max |f(x)| = 1$. Green's formula yields that
\begin{equation*}
\sum_x \widetilde \Delta_\rho f_1(x) \rho(x) = \lambda_1 \sum_x f_1(x) \rho(x) = 0.
\end{equation*}
Indeed, in \eqref{Green} set $f_2 \equiv 1$ and integrate by parts.
Hence,
\begin{equation*}
\sum_x f_1(x) \rho(x) = 0.
\end{equation*}
Therefore,
\begin{align*}
\Re (\rho(x_0)) + \Re \big( \sum_{x \neq x_0} f_1(x) \rho(x) \big) &= 0 \\
\Leftrightarrow \Re (\rho(x_0)) + \sum_{x \neq x_0} \Re f_1(x) \Re \rho(x) - \sum_{x \neq x_0} \Im f_1(x) \Im \rho(x) &= 0,
\end{align*}
but
\begin{equation*}
\big| \sum_{x \neq x_0} \Im (f_1(x)) \Im (\rho(x)) \big| \leq \sum_{x \neq x_0} |\Im \rho(x)|.
\end{equation*}
Hence, by Lemma \ref{lem:LowerBoundAdmittance} and assumption \eqref{eq:LowerBoundRealpart},
\begin{equation*}
\min \Re \rho(x) - \sum_{x \neq x_0} |\Im \rho| > 0,
\end{equation*}
we find that
\begin{equation*}
\sum_{x \neq x_0} \Re (f_1(x)) \Re \rho(x) < 0 .
\end{equation*}
Since $\Re (\rho(x)) > 0$, there is $x_{n+1}$ such that $\Re (f_1(x_{n+1}))<0$.

Another application of Green's formula gives
\begin{equation*}
\lambda_1 \sum |f(x)|^2 \rho(x) = \frac{1}{2} \sum_{x,y} |f(x)-f(y)|^2 \rho_{xy},
\end{equation*}
which implies
\begin{equation*}
|\lambda_1| = \frac{\frac{1}{2} \big| \sum_{x,y} |f(x)-f(y)|^2 \rho_{xy}}{\big| \sum_x |f(x)|^2 \rho(x) \big|}
\end{equation*}
We estimate the numerator from below by finding a path $x_0,\ldots,x_{n+1}$:
\begin{align*}
&\quad \frac{1}{2} \big| \sum_{x,y} |f(x)-f(y)|^2 \rho_{xy} \big| \\
&\geq \frac{1}{2} \sum_{x,y} |f(x)-f(y)|^2 \Re \rho_{xy} \\
&\geq \frac{1}{2} \sum_{x,y} |f(x)-f(y)|^2 \min \Re \rho_{xy} \\
&\geq \sum_{k=0}^n |f(x_k) - f(x_{k+1})|^2 \min \Re \rho_{xy} \\
&\geq \sum_{k=0}^n | \Re f(x_k) - \Re f(x_{k+1}) |^2 \min \Re \rho_{xy} \\
&\geq \frac{\min \Re \rho_{xy}}{n} \big| \sum_{k=0}^n \Re f(x_k) - \Re f(x_{k+1}) \big|^2 \\
&\geq \frac{\min \Re \rho_{xy}}{n} \geq \frac{\Re s \; C_1 \min(|s|^2,|s|^{-4})}{D}.
\end{align*}
For the denominator we find by \eqref{eq:EstimateModulusRho}
\begin{equation*}
\left| \sum_x |f(x)|^2 \rho(x) \right| \leq \sum_x |\rho(x)| \leq C_2 \frac{\min(1,|s|^2)}{\Re s}.
\end{equation*}
This finishes the proof of Theorem \ref{thm:Diameter}.
\end{proof}

\section*{Acknowledgements}
The second author acknowledges financial support by the Deutsche Forschungsgemeinschaft (DFG, Deutsche Forschungsgemeinschaft) Project-ID 258734477 -\\ SFB 1173.

\bibliographystyle{plain}

\end{document}